\documentclass[10pt, a4paper]{article}
\usepackage{amsmath, amsfonts}
\usepackage[usenames]{color}
\usepackage[all]{xy}
\usepackage[latin1]{inputenc}
\usepackage{extarrows}

\usepackage{hyperref} 
\hypersetup{bookmarks = true, 
                  colorlinks = true, 
                  linkcolor = blue, 
                  citecolor = blue, 
                  }
\usepackage{graphicx}
\usepackage{amssymb}
\usepackage{amsmath}
\usepackage{amsthm}

\newtheorem{theorem}{Theorem}[section]

\newtheorem{definition}[theorem]{Definition}
\newtheorem{example}[theorem]{Example}

\newtheorem{lemma}[theorem]{Lemma}

\newtheorem{proposition}[theorem]{Proposition}
\newtheorem{remark}[theorem]{Remark}

\newcommand{\xdownarrow}[1]{%
  {\left\downarrow\vbox to #1{}\right.\kern-\nulldelimiterspace}
}

\title{\textbf{Whitney equisingularity in families of generically reduced curves}}
\author{O.N. Silva\footnote{O.N. Silva: Universidad Nacional Autónoma de México, Instituto de Matemáticas, Unidad Cuernavaca, Av. Universidad s/n, Lomas de Chamilpa, 62210, Cuernavaca, Morelos, Mexico. e-mail: otoniel@im.unam.mx} $\ \ $ and $\ \ $ J. Snoussi\footnote{J. Snoussi: Universidad Nacional Autónoma de México, Instituto de Matemáticas, Unidad Cuernavaca, Av. Universidad s/n, Lomas de Chamilpa, 62210, Cuernavaca, Morelos, Mexico, e-mail: jsnoussi@im.unam.mx}
}
\date{}

\begin{document}

\maketitle

\begin{abstract}

In this work we study equisingularity in a one-parameter flat family of generically reduced curves. We consider some equisingular criteria as topological triviality, Whitney equisingularity and strong simultaneous resolution. In this context, we prove that Whitney equisingularity is equivalent to strong simultaneous resolution and it is also equivalent to the constancy of the Milnor number and the multiplicity of the fibers. These results are extensions to the case of flat deformations of generically reduced curves, of known results on reduced curves. When the family $(X,0)$ is topologically trivial, we also characterize Whitney equisingularity through Cohen-Macaulay property of a certain local ring associated to the parameter space of the family.

\end{abstract}

\section{Introduction}

$ \ \ \ \ \ $ Consider a germ of reduced equidimensional complex surface $(X,0)$ together with an analytic flat map $p$ from $(X,0)$ to $(\mathbb{C},0)$ and consider a representative $p:X\rightarrow T$. The surface $X$ can be viewed as a one parameter flat deformation of the curve $X_0:=p^{-1}(0)$. When $(X,0)$ is not Cohen-Macaulay the germ of curve $(X_0,0)$ has an embedded component at the origin. If $(X_0,0)$ is reduced at all its points $x\neq 0$, then we say that $(X_0,0)$ is a germ of generically reduced curve. We can see that this situation arises in many and natural examples (see \cite[Ex. $55$]{bobadilla2}, \cite[Ex. $2.4$]{jawad}, \cite[Prop. $3.51$]{ref22}, \cite[Ex. $4.6$]{cong} and \cite[Ex. $9.11$]{greuel3}).

Brücker and Greuel studied these deformations in \cite{greuel} and defined the $\delta$ invariant and a Milnor number for a generically reduced curve. In particular, they showed that normalization in family is equivalent to the constancy of $\delta$ of the fibers along the parameter space \cite[Kor. $3.2.1$]{greuel}.

In \cite{cong}, Công-Trình Lê studied topological triviality of such families comparing it to the constancy of Milnor number; the full statement was proved by Greuel in \cite[Th. $9.3$]{greuel3}. 

In \cite[Th. $3.1$]{jawad} Whitney equisingularity of such a family was proved to be equivalent to Zariski's discriminant criterion.

All these results are extensions to the case of flat deformations of generically reduced curves, of known results on reduced curves. See for example, \cite{briancon} and \cite{buch} to know more about the reduced case.

The aim of this work is to study Whitney equisingular deformations of a generically reduced curve. Our first main result is the characterization of Whitney equisingularity by the constancy of the Milnor number and the multiplicity of the fibers along the parameter space (Theorem \ref{whitney1}). The second main result is the equivalence between Whitney equisingularity and strong simultaneous resolution (Theorem \ref{whitney3}).

We also prove that for a topologically trivial deformation of a generically reduced curve, Whitney equisingularity can be characterized by the Cohen Macaulay property of a certain local ring associated to the parameter space (Lemma \ref{whitney2}).

Throughout the paper, $(x_1,\cdots,x_N)$ denotes a local coordinate system for $\mathbb{C}^N$. When we consider a family of curves $(X,0)$ in $(\mathbb{C}^{N+1},0)$ we will denote the last coordinate of $\mathbb{C}^{N+1}$ by $t$. The rings $\mathcal{O}_N\simeq \mathbb{C}\lbrace x_1,\cdots,x_N \rbrace$ and $\mathcal{O}_{N+1}\simeq \mathbb{C}\lbrace x_1,\cdots,x_N,t \rbrace$ denote the local rings of holomorphic functions at the origin, respectively on $\mathbb{C}^N$ and $\mathbb{C}^N \times \mathbb{C} \simeq \mathbb{C}^{N+1}$. 

\section{Families of generically reduced curves}\label{sec2}

$ \ \ \ \ $ We describe here the context in which we will be working all along this paper.

\begin{definition}\label{defsur} (a) Let $(C,0)$ be a germ of curve in $(\mathbb{C}^{N},0)$. We say that $(C,0)$ is generically reduced at $0$ if the only possibly non reduced point of the curve near the origin is the origin itself, that is, a curve with an isolated singularity at the origin.\\ 

\noindent (b) Let $(X,0) \subset (\mathbb{C}^{N+1},0) $ be a germ of a reduced and pure dimensional complex surface. Consider an analytic map $p:(X,0)\rightarrow (\mathbb{C},0)$. We say that the surface $(X,0)$ is a one-parameter flat deformation of the germ of reduced curve $(X_0,0):=(p^{-1}(0),0)$ if $p$ is a flat map. When $(X_0,0)$ is generically reduced, we say $p:(X,0)\rightarrow (\mathbb{C},0)$, or simply $(X,0)$, is a family of generically reduced curves.\\

\noindent (c) Given a representative $p:X\rightarrow T$ of a family of generically reduced curves $(X,0)$, we will denote the fibers of $p$ by $X_t:=p^{-1}(t)$.

\end{definition}

We remark that if $(X,0)$ is not Cohen-Macaulay, then the special fiber $(X_0,0)$ have an embedded component (see \rm\cite[Cor. 6.5.5]{jong}, and also \cite[Th. 17.3]{matsumura}).

Following \cite[p. $248$]{buch} and \cite[Appendix]{greuel3}, we recall briefly the notion of a good representative for a family of generically reduced curves $p:(X,0)\rightarrow (\mathbb{C},0)$.

\begin{definition}\label{goodrepresentative}

Let $(X,0) \subset (\mathbb{C}^{N+1},0) $ be a germ of a reduced and pure dimensional complex surface and consider a flat map $p:(X,0)\rightarrow (\mathbb{C},0)$. \textit{Let $B(0,\epsilon) \subset  \mathbb{C}^N$ (resp. $D(0,\eta) \subset \mathbb{C}$) be an open ball around $0$ of radius $\epsilon$ (resp. an open disc around $0$ of radius $\eta$) and $X_0\subset B(0,\epsilon)$ a closed subspace of $B(0,\epsilon)$ representing $(X_0,0)$.}

\textit{If $T \subset D(0,\eta)$ is a closed ball around $0$ of radius $\eta(\epsilon)$ and $X\subset B(0,\epsilon) \times T$ is a closed subspace representing $(X,0)$, and if $\epsilon$ is sufficiently small and $\eta(\epsilon)$ is sufficiently small with respect to $\epsilon$, then $p:X\rightarrow T$ is called a good representative of $p:(X,0)\rightarrow (\mathbb{C},0)$ (see Figure \rm\ref{fig1}).}

\begin{figure}[h]
\centering
\includegraphics[scale=0.22]{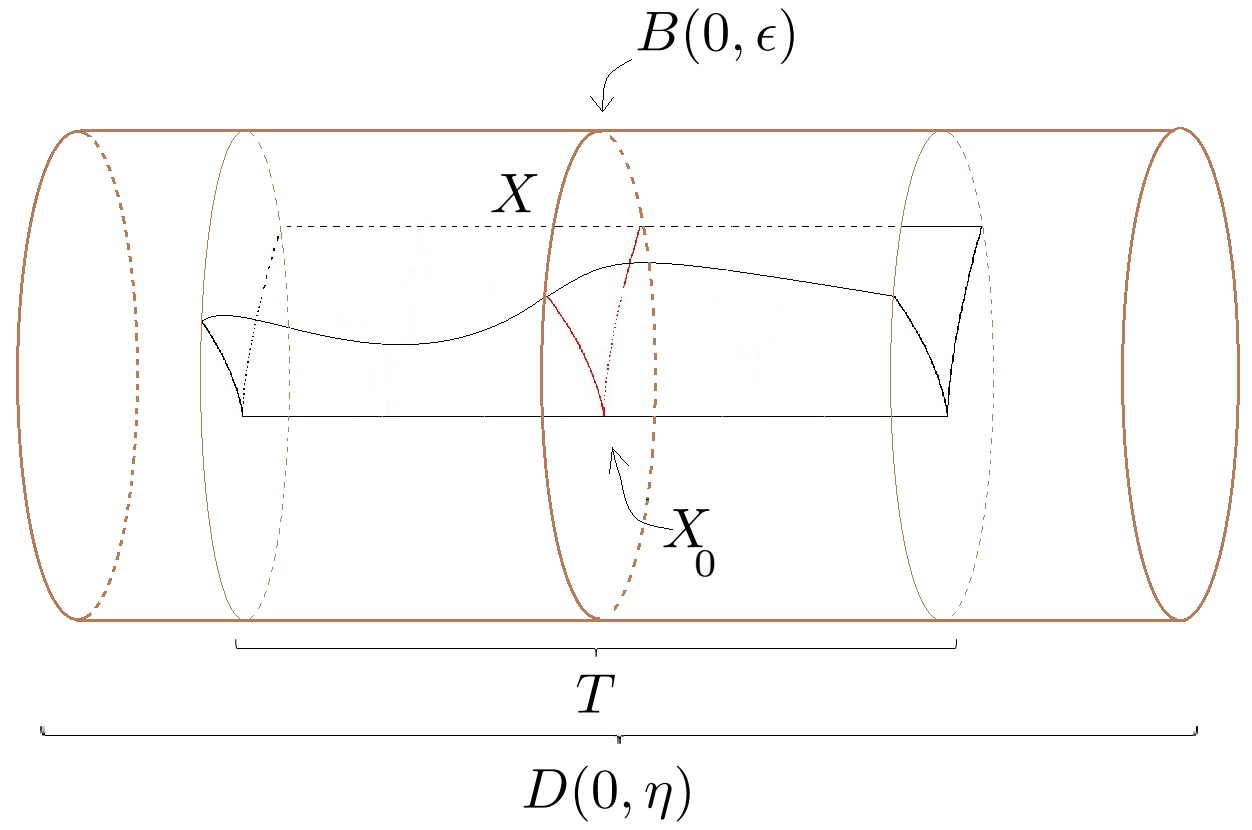}  
\caption{Illustration of a good representative.}\label{fig1}
\end{figure}

\textit{By identifying $B \times \lbrace t \rbrace$ with $B$, we usually consider the fibers $X_t$ of $p$ as subsets of $B$. Talking about good representatives, we allow $\epsilon$ and $\eta(\epsilon)$ to shrink during an argument, if necessary.} \textit{The set $B$ is called a Milnor ball if $\epsilon$ is sufficiently small, and the generic fiber $X_t:=p^{-1}(t)$ ($t \neq 0$) is called Milnor fiber of $p$.}
\end{definition}

We remark that in our setting, given a family of generically reduced curves $p:(X,0)\rightarrow (\mathbb{C},0)$, a good representative $p:X\rightarrow T$ always exists.

The surface $(X,0)$ of the following example appears in \rm\cite[Example $55$]{bobadilla2} as the singular set of a family of surfaces.

\begin{example}\label{exemplo2.1} Let $(X,0) \subset (\mathbb{C}^4,0) $ be the germ of surface parametrized by the map 

$$\begin{array}{rcl}
n:(\mathbb{C}^2,0) & \rightarrow & (X,0) \subset (\mathbb{C}^4,0)\\
(u,t) & \mapsto & (u^3, u^4, tu, t)\\
\end{array}
$$

\noindent it is a germ of reduced and pure dimensional complex surface with an isolated singularity at $0$, defined by the ideal
 
$$I_{(X,0)}= \langle xz-ty, \ y^3-x^4, \ y^2z-tx^3, \ yz^2-t^2x^2, \ z^3-t^3x \rangle \mathbb{C}\lbrace x,y,z,t \rbrace .$$

Consider the linear projection in the last factor $p:(\mathbb{C}^3 \times \mathbb{C},0)\rightarrow (\mathbb{C},0)$. Note that the restriction of $p$ to $(X,0)$ is a flat map, hence $p$ makes the surface into a one parameter deformation of the curve $(X_0,0)$ defined by the ideal: 

$$I_{(X_0,0)}= \langle xz,y^3-x^4,y^2z,yz^2,z^3 \rangle \mathbb{C}\lbrace x,y,z \rbrace .$$

One can see that the curve $(X_0,0)$ has an embedded component at the origin and the surface $(X,0)$ is not Cohen-Macaulay (see also Example \rm\ref{example2.2}\textit{)}. 
\end{example}

\section{The invariants}

$ \ \ \ \ \ $ Following Brücker and Greuel \cite[p. 96]{greuel}, we recall the definitions of the $\delta$-invariant and the Milnor number for 
a generically reduced curve.

\begin{definition}
Let $(C,0) \subset (\mathbb{C}^{N},0)$ be a germ of a generically reduced curve. Denote by $(|C|,0)$ the corresponding curve with reduced structure, and let $n:(\hat{C},\overline{0})\rightarrow (|C|,0)$ be the normalization of $(|C|,0)$. The following numbers:

\begin{center}
$\delta(|C|,0):=dim_{\mathbb{C}} \left(\dfrac{n_{\ast}(\mathcal{O}_{(\hat{C},\overline{0})})}{\mathcal{O}_{(|C|,0)}} \right)$, $\ \ \ $ $\epsilon(C,0):=dim_{\mathbb{C}} H^0(\mathcal{O}_{(C,0)})$ $\ \ \ \ $ and $ \ \ \ \ $ $\delta(C,0):=\delta(|C|,0)-\epsilon(C,0)$\\
\end{center}

\noindent are respectively the delta-invariant of $(|C|,0)$, the epsilon-invariant and the delta-invariant of $(C,0)$, where $H^0(R)$ denotes the $0$-local cohomology of the local ring $R$. Further,

\begin{center}
$m(C,0)$, $ \ \ \ \ $ $\mu(|C|,0)=2\delta(|C|,0)-r(C,0)+1 \ \ \ \ $ and $\ \ \ \ \mu(C,0)=\mu(|C|,0)-2\cdot \epsilon(C,0)$
\end{center}

\noindent are respectively the Hilbert-Samuel multiplicity of the maximal ideal $M({\mathcal{O}_{(C,0)}})$ of the local ring 
$\mathcal{O}_{(C,0)}$ of $C$ at $0$ and the Milnor number of $(|C|,0)$ and $(C,0)$; where $r(C,0)$ denotes the number of branches of $(C,0)$.

\end{definition}

\begin{remark}\label{remark3.2}
Let $(C,0)$ be a germ of generically reduced curve with local ring $\mathcal{O}_{(C,0)}\simeq\dfrac{\mathcal{O}_{N}}{I_{(C,0)}}$. The $\epsilon$ invariant can be also calculated as 

\begin{center}
 $\epsilon(C,0)=dim_{\mathbb{C}} \ (\sqrt{0}) \ \mathcal{O}_{(C,0)}=dim_{\mathbb{C}} \left( \dfrac{\sqrt{{I_{(C,0)}}}}{{I_{(C,0)}}} \right) \mathcal{O}_{N}$,
 \end{center} 

\noindent that is, the dimension of the nilradical of $\mathcal{O}_{(C,0)}$ as a $\mathbb{C}$-vector space (which is finite if and only if $(C,0)$ has at most an isolated singularity at the origin; see for instance \rm\cite[\textit{page 30}]{greuel3}\textit{). Also in} \rm\cite[\textit{Corollary 5.7}]{greuel3}\textit { Greuel showed that if $I=P_1\cap \cdots \cap P_r \cap Q$ is an irredundant primary decomposition of $I$, where $P_j$ are prime ideals and $Q$ is a (not unique) $m$-primary ideal, then}

\begin{center}
$\epsilon(C,0)=dim_{\mathbb{C}} \mathcal{O}_N/Q \ - \ dim_{\mathbb{C}} \mathcal{O}_N/((P_1\cap \cdots \cap P_r)+Q)$.
\end{center}

\end{remark}

\begin{example}\label{example2.2} 
Let us consider the Example \rm \ref{exemplo2.1}. \textit{A primary decomposition of the ideal $I_{(X_0,0)}$ in $\mathcal{O}_3 \simeq \mathbb{C}\lbrace x,y,z \rbrace $ is given by}

\begin{center}
$I_{(X_0,0)}= \langle  xz,y^3-x^4,y^2z,yz^2,z^3 \rangle \mathcal{O}_3 = \langle z,y^3-x^4 \rangle \mathcal{O}_3 \cap  \langle x^4,xz,y^2,yz^2,z^3 \rangle \mathcal{O}_3 .$
\end{center}

\noindent \textit{The local ring of the corresponding reduced curve is $\mathcal{O}_{(|X_0|,0)} \simeq \dfrac{{\mathcal{O}_3}}{ \langle z, y^3 - x^4 \rangle }$, so $\delta(|X_{0}|,0)=3$ and $\mu(|X_0|,0)=6$. We have also that} 

\begin{center}
$ dim_{\mathbb{C}}\dfrac{\mathbb{C}\lbrace x,y,z \rbrace}{\langle x^4,xz,y^2,yz^2,z^3 \rangle}=11$ $ \ \ \ $ and $ \ \ \ $  $dim_{\mathbb{C}}\dfrac{\mathbb{C}\lbrace x,y,z \rbrace}{\langle x^4,xz,y^2,yz^2,z^3, z, y^3-x^4 \rangle}=8$,
\end{center}

\noindent \textit{thus by Remark} \rm\ref{remark3.2} \textit{we have that $\epsilon(X_0,0)=3$. Therefore, $\delta(X_{0},0)=0$ and $\mu(X_{0},0)=0$.}
\end{example}

\section{Topological triviality and Whitney equisingularity}\label{sectiontop}

$ \ \ \ \ $ In this section we will deal with topological triviality and Whitney equisingularity, and we will state and prove our first main result.

\begin{definition}\label{deftoptrivial} Let $p:(X,0) \subset (\mathbb{C}^{N+1},0) \rightarrow (\mathbb{C},0)$ be a family of generically reduced curves, the surface $(X,0)$ being pure dimensional. Consider a good representative $p:X\rightarrow T$ with a section $\sigma:T \rightarrow X$, such that both $\sigma(T)$ and $X_t \setminus \sigma(t)$ are smooth for $t \in T$.\\ 

\noindent \textit{(a) We say that $p: X\rightarrow T$ is topologically trivial (or for simplicity, $X$ is topologically trivial) if there is a homeomorphism $h: X \rightarrow X_{0} \times T$ such that $p=p' \circ h$, where $p' :X_{0} \times T \rightarrow T$ is the projection on the second factor}.\\

\noindent \textit{(b) We say that $p:X\rightarrow T$ is Whitney equisingular (or for simplicity, $X$ is Whitney equisingular) if the stratification} $ \lbrace X \setminus \sigma(T), \sigma(T) \rbrace $ \textit{satisfies Whitney's conditions (a) and (b) at $0$, that is:}\\

\textit{For any sequences of points $(x_n)\subset X \setminus \sigma(T)$ and $(t_n) \subset \sigma(T) \setminus \lbrace 0 \rbrace$ both converging to $0$ and such that the sequence of lines $(x_nt_n)$ converges to a line $l$ and the sequence of directions of tangent spaces, $T_{x_n}X$, to $X$ at $x_n$, converges to a linear space $H$ we have:}

\begin{flushleft}
\textit{(Whitney's condition a): The tangent space to $\sigma(T)$ at $0$ is contained in $H$.}\\
\textit{(Whitney's condition b): The line $l$ is contained in $H$.}
\end{flushleft}

\end{definition}

\begin{remark}\label{remarkwhitney} Let $(X,0) \subset (\mathbb{C}^{N+1},0) $ be a germ of a reduced and pure dimensional complex surface and consider a flat map $p:(X,0)\rightarrow (\mathbb{C},0)$ and a good representative $p:X\rightarrow T$.\\ 

\noindent (a) If the singular locus $\Sigma(X)$ of $X$ is not smooth, then $\lbrace X \setminus \Sigma(X), \Sigma(X) \rbrace$ is not an eligible stratification (with smooth strata) and $\lbrace X \setminus \Sigma(X), \Sigma(X) \setminus  0 , 0 \rbrace$ is a trivial Whitney stratification. In this way, in order to study Whitney's conditions along $\Sigma(X)$ we need to ask $\Sigma(X)$ to be smooth or to be the origin. For investigating equisingularity of a family of generically reduced curves $p:X\rightarrow T$, we assume in Definition \rm\ref{deftoptrivial} \textit{that the Milnor fiber $X_t$ is non singular outside the section $\sigma(T)$. This assumption implies that $\Sigma(X)=\sigma(T)$ or $\Sigma(X)=0$, since $\Sigma(X_t)=X_t \cap \Sigma(X)$ (see} \rm\cite[\textit{Lemma 8.1}]{greuel3}\textit{). So, in our case $\lbrace X \setminus \sigma(T), \sigma(T) \rbrace$ is in fact a stratification of $X$ with smooth strata.}\\
 
\noindent \textit{(b) Now, suppose that the Milnor fiber $X_t$ has another singular point outside the section $\sigma(T)$, again by} \rm\cite[\textit{Lemma 8.1}]{greuel3} \textit{we have the situation known as ``splitting of singularities''. In such a case, the singular locus $\Sigma(X)$ of $X$ is not smooth; so it does not make sense to talk about Whitney's conditions along $\Sigma(X)$. On the other hand, even if the singular locus of $X$ is not smooth, we can still consider topological triviality. However, we will not treat the case where $\Sigma(X)$ is singular in this work.}\\ 

\noindent \textit{(c) We will describe a ``special'' form of the map $p:X\rightarrow T$ that we will use at some specific moments of this work. From (a), we have that $\Sigma(X)=\sigma(T)$ or $\sigma(T)=0$. Thus, after a suitable change of coordinates in $\mathbb{C}^{N+1} \simeq \mathbb{C}^N \times \mathbb{C}$, we can assume that:}

\begin{enumerate}

\item \textit{The set $\Sigma(X)$ is $(0,\cdots,0) \times T$ or is the origin}.
  
\item \textit{The section $\sigma:T\rightarrow X$ is defined as $\sigma(t)=(0,\cdots,0,t)$, that is, $\sigma(T) =\{0\}\times T$}. 

\item \textit{The projection $p:X \rightarrow T$ is the restriction to $X$ of the canonical projection of $\mathbb{C}^N \times T$ on the last factor}.

\end{enumerate} 

\noindent \textit{(d) Consider the decomposition $X=X^1\cup \cdots \cup X^r$ of $X$ into irreducible components. Consider a sequence $(x_n)\subset X \setminus T$ converging to $0$. Note that the sequence could be decomposed into subsequences $(x_{ \tau(a_j(n))})\subset X^j$ for some $j$. Hence we have that $X$ is Whitney equisingular along $\sigma(T)$ if and only if $X^j$ is Whitney equisingular along $\sigma(T)\cap X^j$ for all $j=1,\cdots,r$.}

\end{remark}

\begin{theorem}\label{toptriviality} Let $p:(X,0)\rightarrow (\mathbb{C},0)$ be a family of generically reduced curves and suppose that there is a good representative $p:X\rightarrow T$ with a section $\sigma:T \rightarrow X$, such that both $\sigma(T)$ and $X_t \setminus \sigma(t)$ are smooth for $t \in T$. \textit{The following statements are equivalent:}

\begin{flushleft}
\textit{(1) $X$ is topologically trivial.} 

\textit{(2) $\delta(X_{t}, \sigma(t))$ and $r(X_{t},\sigma(t))$ are constant for all $t\in T$ and $(X,0)$ is equidimensional}.

\textit{(3) $\mu(X_{t},\sigma(t))$ is constant and the Milnor fiber $X_t$ is connected for all $t\in T$.}
\end{flushleft}

\end{theorem}

Theorem \ref{toptriviality} is proved in \cite[Theorem $9.3$]{greuel3}, see also \cite[Theorem $4.5$]{cong}.\\ 

Let $W$ be a complex analytic space. As in \cite[Section 8]{greuel3}, we will denote the number of connected components of $W$ by $b_0(W)$. The following proposition shows us a way of computing the number of connected components of the Milnor fiber. It is stated and proved by Greuel in a more general context.

\begin{proposition}\label{propgreuel} (\rm\cite[\textit{Proposition 8.8}]{greuel3}\textit{) Let $p:(X,0)\rightarrow (\mathbb{C},0)$ be a family of generically reduced curves and consider a good representative $p:X\rightarrow T$. Then $b_0(X_t)=b_0(X \setminus \lbrace 0 \rbrace)$.}
\end{proposition}

The following Example shows that the hypothesis about the connectivity of the Milnor fiber $X_t$ in Theorem \ref{toptriviality}(c) is in fact necessary. 

\begin{example}\label{exgreuel}(\rm\cite[\textit{Example 9.11(4)}]{greuel3}\textit{) Consider the surface $(X,0) \subset (\mathbb{C}^5 \times \mathbb{C},0)$ defined by the following ideal} 

\begin{center}
$ I_{(X,0)} = \langle xz, \ yz, \ xw, \ yw, \ zw-tz, \ w^2-tw, \ xv, \ yv, \ zv, \  wv, \ x^2y+xy^2+tx^2+ty^4 \rangle \mathcal{O}_6  $
\end{center}

\noindent \textit{where $\mathcal{O}_6 \simeq \mathbb{C}\lbrace x,y,z,w,v,t \rbrace$. We have that $(X,0)$ is a pure dimensional reduced germ of surface with $3$ irreducible components $(X^{1},0),(X^2,0)$ and $(X^3,0)$ defined by the following ideals}

\begin{center}
$I_{(X^1,0)} = \langle x^2y+xy^2+t(x^2+y^4), \ z, \ v, \ w \rangle \mathcal{O}_6$, $ \ \ $ $I_{(X^2,0)} = \langle   x, \ y, \ v, \ w-t \rangle \mathcal{O}_6$ $ \ \ $ \textit{and} $ \ \ $ $I_{(X^3,0)} = \langle  x, \ y, \ z, \ w \rangle \mathcal{O}_6$
 \end{center} 

\textit{Note that the restriction to $(X,0)$ of the linear projection $p:(\mathbb{C}^5 \times \mathbb{C},0)\rightarrow (\mathbb{C},0)$ in the last factor is flat. Thus, $p:(X,0)\rightarrow (\mathbb{C},0)$ is a family of generically reduced curves.} \textit{Take a good representative $p:X\rightarrow T$ and consider the section $\sigma:T\rightarrow X$ defined as $\sigma(t)=(0,\cdots,0,t)$. Note that $X_t \setminus \sigma(t)$ is smooth for all $t \in T$, thus we are in the conditions of Theorem} \rm\ref{toptriviality}. 

\textit{We have that $\epsilon(X_0,0)=1$, $\mu(|X_0|,0)=6$, thus $\mu(X_0,0)=4$. In the other hand, $X_t$ is reduced for $t\neq 0$, consisting of an $A_3$-singularity at $\sigma(t)$, a line through $\sigma(t)$ and another disjoint line. Thus, $\mu(X_t,\sigma(t))=4$ for all $t$, but $X_t$ is not connected. Clearly, we have also that $X$ is not topologically trivial, since the number of branches of $(X_t,\sigma(t))$ is not constant along $\sigma(T)$.}

\textit{We remark that the multiplicity of the fibers is not constant along the section $\sigma(T)$, more precisely, $m(X_0,0)=5$ and $m(X_t,\sigma(t))=3$ for $t\neq 0$. Note also that $X^2 \cap (X^1 \cup X^3)=0$, in particular, $X^2 \cap \sigma(T)=0$.}

\end{example}

Inspired in Example \ref{exgreuel} we have the following lemma which is essentially a consequence of Proposition \ref{propgreuel}.

\begin{lemma}\label{lemmaaux2} Let $p:(X,0)\rightarrow (\mathbb{C},0)$ be a family of generically reduced curves; the surface $(X,0)$ being equidimensional. Suppose there is a good representative $p:X\rightarrow T$ with a section $\sigma:T \rightarrow X$, such that both $\sigma(T)$ and $X_t \setminus \sigma(t)$ are smooth for $t \in T$.\\

\noindent (a) If $X$ is irreducible, then the Milnor fiber $X_t$ is connected for all $t \in T$.

\noindent (b) If $m(X_t,\sigma(t))$ is constant for all $t \in T$, then the Milnor fiber $X_t$ is connected for all $t \in T$.
\end{lemma}

\begin{proof} (a) By Proposition \ref{propgreuel} and irreducubility of $(X,0)$, we have $b_0(X_t)=b_0(X \setminus \lbrace 0 \rbrace)=1$, hence $X_t$ is connected for all $t\in T$.\\

\noindent (b) First, let $X= \bigcup _{i=1}^r X^i$ be the decomposition of $X$ into irreducible components. They all are of dimension two. The hypothesis that $X_t$ has at most a singular point at $\sigma(t)$ implies that $X^j \cap X^i=\sigma(T)$ or $X^j \cap X^i= \lbrace 0 \rbrace $ for $i \neq j$.

Let us decompose the set of indices as follows: $\{1, \ldots , r\} = A \cup B$ with $i\in A$ if and only if $\sigma(T) \subset X^i$ and $i\in B$ if and only if $\sigma(T) \cap X^i = \{0\}$. Suppose that $X_t$ is not connected, then by Proposition \ref{propgreuel}, we have that $X \setminus \lbrace 0 \rbrace$ is not connected. However $\sigma(T)\setminus \{0\}$ is connected. So, in this case $B \neq \emptyset$. For any $t\in T$, we can write $X_t= \bigcup_{i=1}^r X_t^i$ with $X_t^i = X_t \cap X^i$, is a generically reduced curve. By additivity of the multiplicity we have: 

\begin{center}
$m( X_t , \sigma(t)) = \displaystyle \sum_{i\in A\cup B}m(X_t^i, \sigma(t)),$
\end{center}

\noindent and for $t\neq 0$ we have:
 
 \begin{equation}\label{eq1}
 m(X_t, \sigma(t)) =  \displaystyle \sum_{i\in A}m(X_t^i, \sigma(t)).
 \end{equation}

The restriction of the projection $p$ to each component $X^i$ is a flat deformation of the curve $X_0^i$. So the multiplicity of any $X_t^i$ is upper semi-continuous. Together with equality (\ref{eq1}) and the fact that the fiber $X_0$ is connected, we have:

$$m(X_t,\sigma(t)) =  \displaystyle \sum_{i\in A} m(X_t^i, \sigma(t))
\leq  \displaystyle \sum_{i\in A} m(X_0^i, 0) < m(X_0,0),$$

\noindent which contradicts the constancy of the multiplicity.\end{proof}

In the reduced case, we have that $X$ is Whitney equisingular if and only if the Milnor number $\mu(X_{t},\sigma(t))$ and the multiplicity $m(X_{t},\sigma(t))$ are constant for all $t \in T$ \cite[Théorème III.3]{briancon}. Now, we show a generalization of this result to the generically reduced case.

\begin{theorem}\label{whitney1} Let $p:(X,0)\rightarrow (\mathbb{C},0)$ be a family of generically reduced curves, with $(X,0)$ reduced and equidimensional. Suppose that there is a good representative $p:X\rightarrow T$ with a section $\sigma:T \rightarrow X$, such that both $\sigma(T)$ and $X_t \setminus \sigma(t)$ are smooth for $t \in T$. \textit{Then the following conditions are equivalent:}

\begin{flushleft}
\textit{(1) $X$ is Whitney equisingular.}\\
\textit{(2) The invariants $\mu(X_{t},\sigma(t))$ and  $m(X_t,\sigma(t))$ are constant for all $t\in T$.}
\end{flushleft}

\end{theorem}

We denote by $e(q,R)$ the Hilbert-Samuel multiplicity of the $M(R)$-primary ideal $q$ in the local ring $(R,M(R))$ (see \cite[p. $107$]{matsumura}), where $M(R)$ denotes the maximal ideal of $R$. If $(R,M(R))$ is a zero-dimensional local ring, then we denote by $l(R,M(R))$ the length of $R$ (see \cite[Chapter $1$]{matsumura}). The key to proving Theorem \ref{whitney1} is the combination of Lemma \ref{lemmaaux2} and the following lemma.

\begin{lemma}\label{lemmaauxiliar}
Let $(C,0)$ be a germ of generically reduced curve with local ring $\mathcal{O}_{(C,0)}\simeq\dfrac{\mathcal{O}_{N}}{I_{(C,0)}}$. Consider $(|C|,0)$ the associated reduced curve, with local ring $\mathcal{O}_{(|C|,0)}\simeq\dfrac{\mathcal{O}_{N}}{\sqrt{I_{(C,0)}}}$. Then $m(C,0)=m(|C|,0)$.

\end{lemma}

\begin{proof}
Let $M(\mathcal{O}_{(C,0)})$ and $M(\mathcal{O}_{(|C|,0)})$ be the maximal ideals of $\mathcal{O}_{(C,0)}$ and $\mathcal{O}_{(|C|,0)}$, respectively. Let $\lbrace P_{1}, \cdots, P_{r} \rbrace$ be the set of all minimal prime ideals of $\mathcal{O}_{(C,0)}$. For all $j=1, \cdots, r$, we have the following commutative diagram

 \[\xymatrix{\mathcal{O}_{(C,0)}\simeq  \dfrac{\mathbb{C}\lbrace x_1,\cdots,x_N \rbrace}{I_{(C,0)}} \ \ \  \ar[r]^{\pi_1} \ar[dd]^{\pi_2}  & \ \ \  \dfrac{\mathbb{C}\lbrace x_1,\cdots,x_N \rbrace}{\sqrt{I_{(C,0)}}}  \simeq \mathcal{O}_{(|C|,0)}  \ar[ddl]^{\pi_3}  \\
             &  \\
             \dfrac{\mathbb{C}\lbrace x_1,\cdots, x_N \rbrace}{P_j}  &  }\]

\noindent where $\pi_{1}, \pi_{2}$ and $\pi_{3}$ are quotient morphisms. We have that $\mathcal{O}_{(C,0)}/P_{j} \simeq \mathcal{O}_{(|C|,0)}/P_{j}$, hence 

\begin{center}
$e \left( M  \left(   \dfrac{\mathcal{O}_{(C,0)}}{P_{j}}   \right)  , \dfrac{\mathcal{O}_{(C,0)}}{P_{j}} \right) =e  \left( M \left( \dfrac{\mathcal{O}_{(|C|,0)}}{P_{j}} \right) , \dfrac{\mathcal{O}_{(|C|,0)}}{P_{j}} \right) $ 
\end{center}

\noindent for all $j=1, \cdots, r$, where $M(\mathcal{O}_{(C,0)}/P_{j})$ and $M (\mathcal{O}_{(|C|,0)}/P_{j})$ are the maximal ideals of $\mathcal{O}_{(C,0)}/P_{j}$ and $\mathcal{O}_{(|C|,0)}/P_{j}$, respectively.

We have that $\mathcal{O}_{(C,0)}$ and $\mathcal{O}_{(|C|,0)}$ satisfy the Serre's condition $R_{0}$ (see \cite[p. $183$]{matsumura}) and the ideal $P_{j}$ has height $0$, hence the localizations ${(\mathcal{O}_{(C,0)})}_{P_{j}}$ and ${(\mathcal{O}_{(|C|,0)})}_{P_{j}}$ are regular local rings and also are integral domains. Since ${(\mathcal{O}_{(C,0)})}_{P_{j}}$ and ${(\mathcal{O}_{(|C|,0)})}_{P_{j}}$ are Artinian, they are fields and we have that ${(\mathcal{O}_{(C,0)})}_{P_{j}}\simeq{(\mathcal{O}_{(|C|,0)})}_{P_{j}}\simeq \mathbb{C}$, hence we obtain for the lengths that $l({(\mathcal{O}_{(C,0)})}_{p_{j}})=l({(\mathcal{O}_{(|C|,0)})}_{P_{j}})=1$. By Serre's formula (see \cite{matsumura}, Th. 14.7) we have that

\begin{center}
$ \ \ \ \ \ \ \ \ \ \ \ \ \ \ \ m(C,0)=\displaystyle { \sum_{j=1}^{r}}e(M (\mathcal{O}_{(C,0)}/P_{j}),\mathcal{O}_{(C,0)}/P_{j})\cdot l({(\mathcal{O}_{(C,0)})}_{P_{j}})= \ \ \ \ \ \ \ \ \ \ \ \ \ $\\

$ \ \ \ \ \ \ \ \ \ \ \ \ \ \ \ \ \ \ \ \ \ \ \ \ \ \ \ \ \ \ \ \ \ \ \ \ \ \ \ \displaystyle { \sum_{j=1}^{r}}e(M (\mathcal{O}_{(|C|,0)}/P_{j}),\mathcal{O}_{(|C|,0)}/P_{j})\cdot l({(\mathcal{O}_{(|C|,0)})}_{P_{j}})=m(|C|,0). \ \ \ \ \ \ \ \ \ \ \ \ \ \ \ \ \ \ \ \ \ \ \ \ \ \  \ \ \ \ \ \ \ \ \ \ \ \qedhere $ 
\end{center}

\end{proof}

We are now able to prove Theorem \ref{whitney1}.

\begin{proof}(of Theorem \ref{whitney1}) By Lemmas \ref{lemmaauxiliar} and \ref{lemmaaux2}, we can prove the statement essentially in the same way as in the reduced case (see \cite{briancon}, Theorem III.3), we include the proof for completeness.\\

\noindent (1) $\Rightarrow$ (2) Since $X$ is Whitney equisingular, then $X$ is topologically trivial and by Theorem \ref{toptriviality} we have that $\mu(X_{t},\sigma(t))$ is constant. The constancy of $m(X_{t},\sigma(t))$ follows by a result of Hironaka; see \cite{hironaka}, Cor. 6.2.
 
\noindent (2) $\Leftarrow$ (1) Let $X^1,\cdots, X^r$ be the irreducible components of $X$ and denote by $X_t^j=p_j^{-1}(t)$ the fibers of $p_j$ where $p_j$ is the restriction of $p$ to $X^j$. By the additivity property of the multiplicity we have that

\begin{center}
$m(X_t,\sigma(t))=m(X_t^1,\sigma(t))+\cdots+ m(X_t^r,\sigma(t))$, $ \ \ $ for all $t$.
\end{center}

Since the multiplicity is an upper semi-continuous invariant, we have that $m(X_t,\sigma(t))$ is constant if and only if $m(X_t^j,\sigma(t))$ is constant for all $j$. 

In the other hand, since $m(X_t,\sigma(t))$ is constant, then by Lemma \ref{lemmaaux2} the Milnor fiber $X_t$ is connected for all $t\in T$. Since the Milnor number $\mu(X_t,\sigma(t))$ is constant, then $X$ is topologically trivial, by Theorem \ref{toptriviality}. This implies that $X^j$ is also topologically trivial for all $j=1,\cdots,r$, hence $\mu(X^j_t,\sigma(t))$ is constant, again by Theorem \ref{toptriviality}. Therefore, since the constancy of $m(X_t,\sigma(t))$ and $\mu(X_t,\sigma(t))$ implies the constancy of $m(X_t^j,\sigma(t))$ and $\mu(X_t^j,\sigma(t)$ for all $j$, by Remark \ref{remarkwhitney} (d), it is sufficient to prove the irreducible case.

Suppose $X$ is irreducible and that $p:X\rightarrow T$ and $\sigma:T\rightarrow X$ are as in Remark \ref{remarkwhitney} (c). In order to prove Whitney condition (b), it is enough to prove condition (a) and that, for every sequence of points $(x_n)$ in $X \setminus \sigma(T) $ converging to the origin, the sequence of lines generated by $x_n$ and $p(x_n)$ converges to a line contained in the limit of tangent spaces $T_{x_n}(X)$ (see \cite[Section 17]{whitney65}).

After a change of coordinates we can suppose also that the axis $x_1$ is tangent to $X_0$. Since $X$ is topologically trivial we have that $X$ admits a weak simultaneous resolution (see \cite[Def. $2.1$ and Th. $9.3$]{greuel3}) $\varphi : D \times T\rightarrow X \subset B \times T$, which can be given as a deformation of a parametrization of $(|X_0|,0)$. That is, $\varphi$ can be defined by 

\begin{center}
$\varphi(u,t):=\left(u^{m}(1+b_1(t)) \ , \ b_2(t)u^m + \displaystyle { \sum_{i>{m}}^{}}(a_{2,i}+b_{2,i}(t))u^i   \  , \ \cdots \ , \ b_N(t)u^m + \displaystyle { \sum_{i>{m}}^{}}(a_{N,i}+b_{N,i}(t))u^i  \ , \ t \right)$
\end{center}

\noindent such that, for all $t$, the restriction $\varphi$ to $D \times \lbrace t \rbrace$ is the normalization map of $|X_t|$, where $a_{j,i} \in \mathbb{C}$, $b_j(t), b_{j,i}(t) \in \mathbb{C} \lbrace t \rbrace$, $b_j(0)=b_{j,i}(0)=0$ for all $j$ and $i$, $m=m(X_t,\sigma(t))$, $D$ is an open ball around $0\in \mathbb{C}$ and $B,T$ are as in Definition \ref{goodrepresentative}. Consider the maps 

\begin{center}
$v(u,t):= \dfrac{\partial \varphi}{\partial u}(u,t)$ $ \ \ \ \ $ and $ \ \ \ \ $ $w(u,t):= \dfrac{\partial \varphi}{\partial t}(u,t)$. 
\end{center}

Let $x_n=\varphi(u_n,t_n)$ be a sequence of points in $X \setminus T$. Note that when $x_n\rightarrow 0$, the limit of tangent planes to $X$ at the points $x_n$ contains the limit of the vectors $ \left( \frac{1}{m(1+b_1(t))u^{m-1}} \right) v(u_n,t_n)$ and $w(u_n,t_n)$, which are precisely:

\begin{center}
$\lim \ \left( \frac{1}{m(1+b_1(t))u^{m-1}} \right)v(u_n,t_n)=(1,0,\cdots,0):=e_1$ $ \ \ \ \  $ and $ \ \ \ \ $ $\lim \ w(u_n,t_n)=(0,\cdots,0,1):=e_t$
\end{center}

\noindent when $(u_n,t_n)\rightarrow 0$. Hence, the limit of tangent planes to $X$ at the points $x_n$ is the plane generated by  the vectors $e_1$ and $e_t$. Since the tangent space to $\sigma(T) = T $ at $0$ is the line generated by the vector $e_t$, then Whitney condition (a) is satisfied. Note that

\begin{center}
$\overline{x_n p(x_n)}=\left(u^{m}(1+b_1(t)) \ , \ b_2(t)u^m + \displaystyle { \sum_{i>{m}}^{}}(a_{2,i}+b_{2,i}(t))u^i   \  , \ \cdots \ , \ b_N(t)u^m + \displaystyle { \sum_{i>{m}}^{}}(a_{N,i}+b_{N,i}(t))u^i  \ , \ 0 \right)$.
\end{center}

Hence, we have that $\lim \ \left(\dfrac{1}{u^{m}(1+b_1(t))}\right) \overline{x_n p(x_n)} =e_1$ when $x_n, p(x_n)\rightarrow 0$, therefore Whitney condition (b) is also satisfied.\end{proof}

\begin{example}\label{exemplo3.9} (a) Let us consider the Example \rm \ref{exemplo2.1}. \textit{Take a good representative $p:X\rightarrow T$ of $p:(X,0)\rightarrow (\mathbb{C},0)$ and consider the section $\sigma:T\rightarrow X$ defined as $\sigma(t)=(0,0,0,t)$. Note that $X_t \setminus \sigma(t)$ is smooth for all $t \in T$. We have that $\mu(X_t,\sigma(t))$ is constant and since $X$ is irreducible, $(X_t,\sigma(t))$ is connected for all $t$ by Lemma} \rm\ref{lemmaaux2}. \textit{ Hence, by Theorem} \rm\ref{toptriviality} \textit{ $X$ is topologically trivial. Take $t_0 \in T$, note that the restriction $n_{t_0}(u)=(u^3,u^4,t_0u)$ of the map}

$$\begin{array}{rcl}
n:(\mathbb{C}^2,0) & \rightarrow & (X,0) \subset (\mathbb{C}^4,0)\\
(u,t) & \mapsto & (u^3, u^4, tu, t)\\
\end{array}
$$

\noindent \textit{to $(\mathbb{C} \times \lbrace t_0 \rbrace,0)$ is a parametrization of $(|X_{t_0}|,\sigma(t_0))$; here we are considering the identifications $\mathbb{C} \times \lbrace t_0 \rbrace \simeq \mathbb{C}$ and $\mathbb{C}^3 \times \lbrace t_0 \rbrace \simeq \mathbb{C}^3$. From the parametrization, we see that $m(X_0,0)=3$ and $m(X_t,\sigma(t))=1$ for $t\neq 0$, therefore $X$ is not Whitney equisingular, by Theorem} \rm\ref{whitney1}.\\

\noindent \textit{(b) Let $(X,0)$ be the germ of surface parametrized by the map}

$$\begin{array}{rcl}
n:(\mathbb{C}^2,0) & \rightarrow & (X,0) \subset (\mathbb{C}^4,0)\\
(u,t) & \mapsto & (u^3, u^4, tu^5, t)\\
\end{array}
$$

\noindent \textit{it is a reduced surface defined by the ideal}
 
$$I_{(X,0)}= \langle xz-ty^2,yz-tx^3,z^2-t^2x^2y,y^3-x^4 \rangle \mathcal{O}_4.$$

\textit{Again the projection $p$ to the last factor makes the surface into a one parameter deformation of the curve $(X_0,0)$ defined by the ideal:} 

$$I_{(X_0,0)}= \langle xz,yz,z^2,y^3-x^4 \rangle \mathcal{O}_3  = \langle z,y^3-x^4 \rangle \mathcal{O}_3 \cap \langle x^4,xz,y,z^2 \rangle \mathcal{O}_3.$$

\textit{One can see that the curve $(X_0,0)$ has an embedded component at the origin, so $\mathcal{O}_{(X_0,0)}$ is not Cohen-Macaulay, (}\rm\cite[\textit{Corollary 6.5.5}]{jong}\textit{). Take a good representatice $p:X\rightarrow T$ and the section $\sigma:T\rightarrow X$ defined as $\sigma(t)=(0,0,0,t)$. We have that $X_t \setminus \sigma(t)$ is smooth. As we did in the previous example, we can check that $\mu(X_t,\sigma(t))=4$ and $m(X_t,\sigma(t))=3$ for all $t \in T$. Hence, by Theorem} \rm\ref{whitney1}, \textit{$X$ is Whitney equisingular. We remark that this example appears in} \rm\cite[\textit{Example 2.4}]{jawad} \textit{where another proof of Whitney equisingularity is presented.}
\end{example}

\section{Strong simultaneous resolution}

$ \ \ \ \ $ In this section, we compare Whitney equisingularity to strong simultaneous resolution.

\begin{definition} Let $p:(X,0)\rightarrow (\mathbb{C},0)$ be a family of generically reduced curves. Suppose $p:X\rightarrow T$ is a good representative with a section $\sigma:T \rightarrow X$, such that both $\sigma(T)$ and $X_t \setminus \sigma(t)$ are smooth for $t \in T$. \textit{Let $n: \widehat{X}\rightarrow X$ be the normalization of $X$. Denote $\widehat{p}:=p \circ n: \widehat{X}\rightarrow T$. We say that $p$ admits a strong simultaneous resolution if $X$ is topologically trivial and}

\begin{center}
\textit{$n^{-1}(\sigma(T))\cong n^{-1}(0) \times T \ \ \ $ (over $T$).}
\end{center}

\textit{In others words, $n^{-1}(T)\cong n^{-1}(0) \times T$, using that $\sigma(T)\simeq T$. Here, we consider the spaces $n^{-1}(T)$ and $n^{-1}(0) \times T$ are isomorphic with the analytic structure inherited from the normalization map.}

\end{definition}

In the reduced case, we have that $X$ is Whitney equisingular if and only if $X$ admits a strong simultaneous resolution. Now, we show that this result also holds for the generically reduced case.

\begin{theorem}\label{whitney3}
Let $p:(X,0)\rightarrow (\mathbb{C},0)$ be a family of generically reduced curves with $(X,0)$ equidimensional. Suppose $p:X\rightarrow T$ is a good representative with a section $\sigma:T \rightarrow X$, such that both $\sigma(T)$ and $X_t \setminus \sigma(t)$ are smooth for $t \in T$. \textit{Then the following conditions are equivalent:}

\begin{flushleft}
\textit{(1) $(X,0)$ is Whitney equisingular;}\\
\textit{(2) $p:(X,0)\rightarrow(\mathbb{C},0)$ admits a strong simultaneous resolution.}
\end{flushleft}

\end{theorem}

To prove this theorem, we need the following lemma:

\begin{lemma}\label{whitney2}
Let $p:(X,0)\rightarrow (\mathbb{C},0)$ be a family of generically reduced curves and suppose that there is a good representative $p:X\rightarrow T$ with a section $\sigma:T \rightarrow X$, such that both $\sigma(T)$ and $X_t \setminus \sigma(t)$ are smooth for $t \in T$. Assume that $(X,0) \subset (\mathbb{C}^N \times \mathbb{C},0)$ is topologically trivial and that $p$ and $\sigma$ are as in Remark \ref{remarkwhitney} (c). Let $(X^{j},0)$ be an irreducible component of $(X,0)$. Then after a suitable change of coordinates in the source of the normalization map $n_{j}: (\mathbb{C}^2,0)\rightarrow (X^{j},0)$ of $(X^{j},0)$ we have that:

\begin{flushleft}
\textit{(a) the image $\overline{t}$ of $t$ in $\mathcal{O}_{2}/ \langle n_{j}^{\ast}(x_{1}, \cdots, x_{N}) \rangle $ by the canonical projection $\mathcal{O}_2\twoheadrightarrow \mathcal{O}_{2}/ \langle n_{j}^{\ast}(x_{1}, \cdots, x_{N}) \rangle $ is a parameter, and}
\end{flushleft}

\begin{center}
\textit{$ \displaystyle { \sum_{j} \  l \left( \dfrac{\mathcal{O}_{2}}{ \langle n_{j}^{\ast}(x_{1}, \cdots, x_{N}),\overline{t}  \rangle} \right)} = m(X_{0}, \sigma(0))$ $\hspace{0.5cm} $ and $ \hspace{0.5cm} \displaystyle { \sum_{j} \ e \left( (\overline{t}), \dfrac{\mathcal{O}_{2}}{ \langle n_{j}^{\ast}(x_{1}, \cdots, x_{N}) \rangle } \right)}=m(X_{t},\sigma(t))$ for $t \neq 0$.}
\end{center}

\begin{flushleft}
\textit{(b) $X$ is Whitney equisingular if and only if $\mathcal{O}_{2}/ \langle n_{j}^{\ast}(x_{1}, \cdots, x_{N}) \rangle $ is Cohen-Macaulay for all $j=1, \cdots, r$, where $r$ is the number of irreducible components of $X$.}
\end{flushleft}

\end{lemma}

\begin{proof}(a) Suppose first that $X$ is irreducible. Since $X$ is topologically trivial, we have that $X$ admits a weak simultaneous resolution (see \cite[Def. $2.1$ and Th. $9.3$]{greuel3}) $n : D \times T\rightarrow X \subset B \times T$ which is also the normalization of $(X,0)$. Also, $n$ can be given as a deformation of a parametrization of $(|X_0|,0)$. That is, $n$ can be defined by

\begin{center}
$n(u,t)=(\alpha_1(u)u^{s_1}+h_1(u,t),\cdots,\alpha_N(u)u^{s_N}+h_N(u,t),t)$ 
\end{center}

\noindent where $h_j(u,t) \in \mathbb{C}\lbrace u,t \rbrace$, $h_j(u,0)=0$, $\alpha_j(u) \in \mathbb{C}\lbrace u \rbrace$, $\alpha_j(0)\neq 0$ and $s_1\leq \cdots \leq s_N$. Note that $n(u,0)=(\alpha_1(u)u^{s_1},\cdots,\alpha_N(u)u^{s_N},0)$ is the normalization map of $(|X_0|,0)$, hence by Lemma \ref{lemmaauxiliar} we have that $m(X_0,0)=m(|X_0|,0)=s_1$.

Since $n$ is a weak simultaneous resolution the reduced spaces $|n^{-1}(\sigma(T))|$ and $|n^{-1}(0)| \times T$ are isomorphic. After a change of coordinates we can suppose that $n^{-1}(\sigma(T))= 0 \times T$ (as sets). Hence, $\sqrt{ \langle n^{\ast}(x_{1}, \cdots, x_{N}) \rangle } \mathcal{O}_2 =  \langle u \rangle \mathcal{O}_2$, where $\mathcal{O}_2 \simeq \mathbb{C}\lbrace u,t \rbrace$. This implies that $\overline{t}$ is a parameter in $\mathcal{O}_{2}/ \langle n^{\ast}(x_{1}, \cdots, x_{N}) \rangle $. Therefore

\begin{center}
$l \left( \dfrac{\mathcal{O}_{2}}{ \langle n^{\ast}(x_{1}, \cdots, x_{N}),\overline{t} \rangle } \right)=l \left( \dfrac{\mathbb{C}\lbrace u \rbrace}{ \langle \alpha_{1}(u) \cdot u^{s_{1}}, \cdots, \alpha_{N}(u) \cdot u^{s_{N}} \rangle } \right)=l \left( \dfrac{\mathbb{C}\lbrace u \rbrace}{ \langle u^{s_{1}} \rangle } \right)=s_{1}=m(X_{0},\sigma(0))$.
\end{center}

Since $X$ is topologically trivial, the ideal $(\overline{u})$ is the only minimal prime ideal of the ring $\mathcal{O}_2/ \langle n^{\ast}(x_{1}, \cdots, x_{N}) \rangle $, where $\overline{u}$ denotes the image of $u$ in $\mathcal{O}_{2}/ \langle n^{\ast}(x_{1}, \cdots, x_{N}) \rangle $ by the canonical projection $\mathcal{O}_2\twoheadrightarrow \mathcal{O}_{2}/ \langle n^{\ast} ( x_{1}, \cdots, x_{N}) \rangle $. Then, by (\cite{matsumura}, Th. 14.7) we have that

\begin{center}
$ e \left( \langle \overline{t} \rangle , \dfrac{\mathcal{O}_{2}}{ \langle n^{\ast}(x_{1}, \cdots, x_{N}) \rangle } \right)=e\left( \langle \overline{\overline{t}} \rangle ,\dfrac{\mathcal{O}_{2}}{ \langle n^{\ast}(x_{1}, \cdots, x_{N}) ,\overline{u} \rangle }\right)\cdot l \left( \left(\dfrac{\mathcal{O}_{2}}{ \langle  n^{\ast}(x_{1}, \cdots, x_{N}) \rangle }\right)\langle \overline{u} \rangle \right)$.
\end{center}

\noindent where $\overline{\overline{t}}$ is the image of $\overline{t}$ by the canonical projection $\mathcal{O}_{2}/ \langle n^{\ast}(x_{1}, \cdots, x_{N}) \rangle \twoheadrightarrow  \dfrac{\mathcal{O}_{2}/ \langle n^{\ast}(x_{1}, \cdots, x_{N}) \rangle}{\langle \overline{u} \rangle}$. Note that

\begin{center}
$ e\left( \langle \overline{\overline{t}} \rangle ,\dfrac{\mathcal{O}_{2}}{ \langle n^{\ast}(x_{1}, \cdots, x_{N}),\overline{u}  \rangle}\right)=e( \langle  t \rangle ,\mathbb{C}\lbrace t \rbrace)=1$
\end{center}

Let $\Gamma$ be the set of all $i$ such that $h_i(u,t)$ is not identically the zero function. Now, for each $i=1,\cdots, N$ such that $i \in \Gamma$, we can write $h_i(u,t)=u^{\beta_i}\tilde{h}(u,t)$ where $\tilde{h}(0,t)\neq 0$. Set $k:= \ min \lbrace s_{1}, \cdots, s_{N}, \beta_{i} \ | \ i \in \Gamma  \rbrace$, hence we have that $m(X_{t}, \sigma(t))=k$, for $t \neq 0$. Furthermore,

\begin{center}
$l \left( \left(\dfrac{\mathcal{O}_{2}}{ \langle n^{\ast}(x_{1}, \cdots, x_{N}) \rangle }\right) \langle \overline{u} \rangle \right)=l \left( \left(\dfrac{\mathbb{C}\lbrace u,t \rbrace}{ \langle u^{k} \rangle}\right) \langle u \rangle \right)=k=m(X_{t}, \sigma(t))$.
\end{center}

Since the multiplicity of $X_{t}$ at $\sigma(t)$ is the sum of the multiplicities of $X_{t}^{1}, \cdots, X_{t}^{r} $ at $\sigma(t)$, the result for the general case where $X=X^1\cup \cdots \cup X^r$ also holds, where $X_t^i= X_t\cap X^i$.\\

\noindent (b) By item (a) we have that the ideal $ \langle \overline{t} \rangle $ is a parameter ideal in $\mathcal{O}_{2}/ \langle n_{j}^{\ast}(x_{1}, \cdots x_{N}) \rangle $. Since $X$ is topologically trivial, the following statements are equivalent, by \cite[Th. 17.11]{matsumura} together with Theorem \ref{whitney1} and Theorem \ref{toptriviality}:\\

\noindent 1.  $\mathcal{O}_{2}/ \langle n_{j}^{\ast}(x_{1}, \cdots x_{N}) \rangle $ is Cohen-Macaulay for all $j$.\\
2. $ l \left( \dfrac{\mathcal{O}_{2}}{ \langle (n_{j}^{\ast}(x_{1}, \cdots, x_{N}),\overline{t}) \rangle } \right) \ = \ e \left( \langle \overline{t} \rangle , \dfrac{\mathcal{O}_{2}}{ \langle n_{j}^{\ast}(x_{1}, \cdots, x_{N}) \rangle } \right)$ for all $j$.\\
3.  $m(X_{0}^j,\sigma(0))=m(X_{t}^j,\sigma(t))$ for all $j$.\\
4. $X^{j}$ is Whitney equisingular for all $j$.\\ 
5.  $X$ is Whitney equisingular.\end{proof}

\begin{proof}(of theorem \ref{whitney3}) We will suppose that $p:X\rightarrow T$ and $\sigma:T\rightarrow X$ are as in Remark \ref{remarkwhitney} (c).\\

(1) $\Rightarrow$ (2) Let $X^{j}$ be an irreducible component of $X$. It is sufficient to prove that $X^{j}$ admits a strong simultaneous resolution for all $j$, thus we can assume that $X$ is irreducible. Since $X$ is Whitney equisingular, in particular, $X$ admits a weak simultaneous resolution (see \cite[Th. $9.3$]{greuel3})

\begin{center}
$n:(\mathbb{C},0) \times T \rightarrow X \subset (\mathbb{C}^N \times T,0)$
\end{center}

\noindent defined by $n(u,t):=(z_{1}(u,t), \cdots, z_{N}(u,t),t)$, where $z_{i}(u,t) \in \mathbb{C} \lbrace u,t \rbrace$. Let $(x_{1}, \cdots, x_{N},t)$ be a system of coordinates of $\mathbb{C}^N \times \mathbb{C}$. The inverse image of $T$ by $n$ is defined in a neighborhood of $0$ by the ideal $ \langle z_{1}(u,t), \cdots, z_{N}(u,t) \rangle  \mathcal{O}_2 = \langle n^{\ast}(x_{1}, \cdots, x_{N}) \rangle \mathcal{O}_2$.

Since $X$ is topologically trivial we have that $\sqrt{ \langle n^{\ast}(x_{1}, \cdots, x_{N}) \rangle }\mathcal{O}_2= \langle u \rangle \mathcal{O}_2$, thus $ \langle n^{\ast}(x_{1}, \cdots, x_{N})\mathcal{O}_2 \rangle $ is a $ \langle u \rangle $-primary ideal. Since $X$ is also Whitney equisingular, by Lemma \ref{whitney2} (b) we have that $\mathcal{O}_{2}/ \langle n^{\ast}(x_{1}, \cdots, x_{N} \rangle$ is Cohen-Macaulay, so $ \langle n^{\ast}(x_{1}, \cdots, x_{N}) \rangle $ does not have embedded associated primes. Then it follows that $ \langle n^{\ast}(x_{1}, \cdots, x_{N}) \rangle \mathcal{O}_2= \langle u^{m} \rangle\mathcal{O}_2$, where $m$ is the multiplicity of $(X_{0},0)$. The statement follows by the isomorphism of $\mathbb{C}$-algebras 

\begin{center}
$\mathcal{O}_{n^{-1}(0)\times T}\simeq \dfrac{\mathbb{C}\lbrace u,t \rbrace}{ \langle n^{\ast}(x_{1}, \cdots, x_{N},t) \rangle \mathbb{C}\lbrace u,t \rbrace}\otimes_{\mathbb{C}}\mathbb{C}\lbrace w \rbrace \simeq  \dfrac{\mathbb{C}\lbrace u,w \rbrace}{ \langle u^m \rangle \mathbb{C}\lbrace u,w \rbrace} \simeq \dfrac{\mathbb{C}\lbrace u,t \rbrace}{ \langle u^m \rangle \mathbb{C}\lbrace u,t \rbrace}\simeq \mathcal{O}_{n^{-1}(T)}$
\end{center}

\noindent (2) $\Rightarrow$ (1) Again we can assume that $X$ is irreducible. Suppose that $X$ admits a strong simultaneous resolution, then $X$ is topologically trivial and $n^{-1}(T)$ is isomorphic to $n^{-1}(0) \times T$, this implies that

\begin{center}
$\mathcal{O}_{n^{-1}(0)\times T}\simeq \dfrac{\mathbb{C}\lbrace u,t \rbrace}{ \langle n^{\ast}(x_{1}, \cdots, x_{N}),t \rangle\mathbb{C}\lbrace u,t \rbrace}\otimes_{\mathbb{C}}\mathbb{C}\lbrace w \rbrace\simeq \dfrac{\mathbb{C}\lbrace u,t \rbrace}{  \langle n^{\ast}(x_{1}, \cdots, x_{N}) \rangle \mathbb{C}\lbrace u,t \rbrace}\simeq \mathcal{O}_{n^{-1}(T)}$
\end{center}

\noindent Since $\mathcal{O}_{n^{-1}(0)\times T}$ is Cohen-Macaulay, we have that $\mathcal{O}_{2}/ \langle n^{\ast}(x_{1}, \cdots x_{N}) \rangle $ is also Cohen-Macaulay, thus the result follows by Lemma \ref{whitney2}.\end{proof}

\begin{example}\label{exemplo3.8} (a) Let us consider the Example \rm \ref{exemplo2.1}. \textit{Note that the local ring}

\begin{center}
\textit{$\dfrac{\mathbb{C}\lbrace u,t \rbrace}{ \langle n^{\ast}(x,y,z) \rangle } \simeq  \dfrac{\mathbb{C}\lbrace u,t \rbrace}{  \langle u^3,u^4,tu \rangle }\simeq \dfrac{\mathbb{C}\lbrace u,t \rbrace}{ \langle u^3,tu \rangle }$}
\end{center}

\noindent \textit{is not Cohen-Macaulay. Since $(X,0)$ is topologically trivial, by Lemma} \rm\ref{whitney2}\textit{(b) $X$ is not Whitney equisingular and, by Theorem} \rm\ref{whitney3}, \textit{$X$ does not admit a strong simultaneous resolution}.\\

\noindent \textit{(b) Let us consider the Example} \rm \ref{exemplo3.9}\textit{(b). Note that the local ring}

\begin{center}
\textit{$\dfrac{\mathbb{C}\lbrace u,t \rbrace}{ \langle n^{\ast}(x,y,z)  \rangle} \simeq  \dfrac{\mathbb{C}\lbrace u,t \rbrace}{ \langle u^3,u^4,tu^5 \rangle }\simeq \dfrac{\mathbb{C}\lbrace u,t \rbrace}{ \langle u^3 \rangle }$}
\end{center}

\noindent \textit{is Cohen-Macaulay. Since $(X,0)$ is topologically trivial, again by Lemma} \rm\ref{whitney2}\textit{(b) $X$ is  Whitney equisingular and, by Theorem} \rm\ref{whitney3}, \textit{$X$ admits a strong simultaneous resolution}.\\

\textit{Item (a) can be generalized in the following way (see} \rm\cite[\textit{Prop. 3.51}]{ref22}\textit{):}\\

\noindent \textit{(c) Let $s$ be an integer which is not a multiple of $3$. Let $(X,0)$ be the germ of surface parametrized by the map} 

$$\begin{array}{rcl}
n:(\mathbb{C}^2,0) & \rightarrow & (X,0) \subset (\mathbb{C}^4,0)\\
(u,t) & \mapsto & (u^3, u^s, tu, t)\\
\end{array}
$$

The restriction to $(X,0)$ of the canonical projection $p$ to the last factor of $\mathbb{C}^3 \times \mathbb{C}$ makes the surface into a one parameter deformation of a curve $(X_0,0)$. Take a good representative $p:X\rightarrow T$ and let $\sigma: T\rightarrow X$ be the section defined by $\sigma(t)=(0,0,0,t)$. Note that $X_t \setminus \sigma(t)$ is smooth for all $t\in T$.\\

\noindent \textit{(c.1) If $s=3k+1$, then $(X,0)$ is a reduced surface defined by the ideal}
 
$$I_{(X,0)}= \langle x^kz-ty, \ y^3-x^{3k+1}, \ yz^2-t^2x^{k+1}, \ z^3-t^3x, \ y^2z-tx^{2k+1} \rangle \mathcal{O}_4 .$$

\textit{The curve $(X_0,0)$ is defined by the ideal:}
 
$$I_{(X_0,0)}= \langle x^kz,y^3-x^{3k+1},yz^2,z^3,y^2z\rangle  \mathcal{O}_3 .$$

\textit{By Remark} \rm\ref{remark3.2} \textit{we have that} 

$$\epsilon(X_{0},0)= 
{\rm dim}_{{\mathbb C}} \left( \frac{ \langle  z,y^3-x^{3k+1} \rangle }{ \langle x^kz,y^3-x^{3k+1},y^2z,yz^2,z^3 \rangle } \right) \mathcal{O}_3 = 3k,$$

\noindent \textit{since $x^az,x^ayz$ and $x^az^2$, where $a=0,1,\cdots,k-1$, generate the nilradical of $\mathcal{O}_{(X_0,0)}$ as a $\mathbb{C}$-vector space.}

\textit{We have that, $\mu(|X_{0}|,0)=6k$, hence $\mu(X_{0},0)=0$. As the fiber $(X_t,\sigma(t))$ is smooth (and reduced) for $t\neq 0$ we have that $\mu(X_t,\sigma(t))$ is constant for all $t$. Since $X$ is irreducible, by Lemma} \rm\ref{lemmaaux2} \textit{(a) $(X_t,\sigma(t))$ is connected for all $t$. Hence, by Theorem} \rm\ref{toptriviality}, \textit{we have that $X$ is topologically trivial. However, $m(X_0,0)=3$ and $m(X_t,\sigma(t))=1$ for $t\neq 0$, hence $X$ is not Whitney equisingular, by Theorem} \rm\ref{whitney1}.\\

\noindent \textit{(c.2) If $s=3k+2$, then $(X,0)$ is a reduced surface defined by the ideal}
 
$$I_{(X,0)}= \langle yz-tx^{k+1}, \ z^3-t^3x, \ y^3-x^{3k+2}, \ x^{2k+1}z-ty^2, \ x^kz^2-t^2y \rangle \mathcal{O}_4 .$$

\textit{The curve $(X_0,0)$ defined by the ideal:} 

$$I_{(X_0,0)}= \langle yz,z^3,y^3-x^{3k+2},x^{2k+1}z,x^kz^2 \rangle \mathcal{O}_3 .$$

\textit{Again by Remark} \rm\ref{remark3.2} \textit{we have that} 

$$\epsilon(X_{0},0)= 
{\rm dim}_{{\mathbb C}} \left( \frac{ \langle z,y^3-x^{3k+2} \rangle }{ \langle yz,z^3,y^3-x^{3k+2},x^{2k+1}z,x^kz^2 \rangle } \right) \mathcal{O}_3 = 3k+1,$$

\noindent \textit{since $x^az$ and $x^bz^2$, where $a=0,1,\cdots,2k$ and $b=0,1,\cdots,k-1$, generate the nilradical of $\mathcal{O}_{(X_0,0)}$ as a $\mathbb{C}$-vector space.}

\textit{We have that, $\mu(|X_{0}|,0)=6k+2$, hence $\mu(X_{0},0)=0$. Again the fiber $(X_t,\sigma(t))$ is smooth (and reduced) for $t\neq 0$, so we have that $\mu(X_t,\sigma(t))$ is constant for all $t$. We have also that $(X_t,\sigma(t))$ is connected for all $t$. Hence, by Theorem} \rm\ref{toptriviality} \textit{we have that $X$ is topologically trivial. However, $m(X_0,0)=3$ and $m(X_t,\sigma(t))=1$ for $t\neq 0$, so $X$ is not Whitney equisingular, by Theorem} \rm\ref{whitney1}.

\end{example}

\begin{remark} We note that the surface $(X,0)$ of Example \rm\ref{exemplo3.8} \textit{(c) can also be seen as a toric surface parameterized by a monomial finitely determined map germ from $\mathbb{C}^2$ to $\mathbb{C}^4$ (see} \rm\cite[\textit{Theorem 1}]{elenice}\textit{). Thus, using the results in} \rm\cite{elenice}\textit{, one can produce other examples of this kind. For the computations in the examples we have made use of the software Singular} \rm\cite{ref21}.
\end{remark}

\begin{flushleft}
\textit{Acknowlegments:} The authors warmly thank the referee for very careful reading and valuable comments and suggestions.  We would like to thank M.A.S Ruas for many helpful conversations and G-M. Greuel for his suggestions and comments on this work. The first author would like to thank CONACyT for the financial support by Fordecyt 265667. Both authors are grateful to UNAM/DGAPA for support by PAPIIT IN $113817$, and to CONACyT grant 282937.
\end{flushleft}

\end{document}